\documentclass[12pt]{amsart}
\usepackage{amsmath,amssymb,amsbsy,amsfonts,latexsym,amsopn,amstext,
                                               amsxtra,euscript,amscd}
\usepackage{url}
\usepackage[colorlinks,linkcolor=blue,anchorcolor=blue,citecolor=blue]{hyperref}
\usepackage{color}
\begin{document}

\newtheorem{theorem}{Theorem}
\newtheorem{lemma}[theorem]{Lemma}
\newtheorem{claim}[theorem]{Claim}
\newtheorem{cor}[theorem]{Corollary}
\newtheorem{prop}[theorem]{Proposition}
\newtheorem{definition}{Definition}
\newtheorem{question}[theorem]{Question}
\newcommand{\hh}{{{\mathrm h}}}

\def\sssum{\mathop{\sum\!\sum\!\sum}}
\def\ssum{\mathop{\sum\, \sum}}

\def\squareforqed{\hbox{\rlap{$\sqcap$}$\sqcup$}}
\def\qed{\ifmmode\squareforqed\else{\unskip\nobreak\hfil
\penalty50\hskip1em\null\nobreak\hfil\squareforqed
\parfillskip=0pt\finalhyphendemerits=0\endgraf}\fi}%%

%  use the AMS-Euler Fraktur fonts
%%%%%%%%%%%%%%%%%%%%%%%%%%%%%%%%%%
\newfont{\teneufm}{eufm10}
\newfont{\seveneufm}{eufm7}
\newfont{\fiveeufm}{eufm5}
%%%%%%%%%%%%%%%%%%%%%%%%%%%%%%%%%
%
%  allow automatic size selection in math mode
%
%%%%%%%%%%%%%%%%%%%%%%%%%%%%%%%%%
\newfam\eufmfam
     \textfont\eufmfam=\teneufm
\scriptfont\eufmfam=\seveneufm
     \scriptscriptfont\eufmfam=\fiveeufm
%%%%%%%%%%%%%%%%%%%%%%%%%%%%%%%%%
%
%  \frak works on a single symbol at a time...
%
\def\frak#1{{\fam\eufmfam\relax#1}}

\def\fR{\mathfrak R}
\def\fT{\mathfrak{T}}

\def\fA{{\mathfrak A}}
\def\fB{{\mathfrak B}}
\def\fC{{\mathfrak C}}

\def\eqref#1{(\ref{#1})}

\def\vec#1{\mathbf{#1}}

\def\squareforqed{\hbox{\rlap{$\sqcap$}$\sqcup$}}
\def\qed{\ifmmode\squareforqed\else{\unskip\nobreak\hfil
\penalty50\hskip1em\null\nobreak\hfil\squareforqed
\parfillskip=0pt\finalhyphendemerits=0\endgraf}\fi}

%%%%%%%%%%%%%%%%%%%%%%%%%
% Alphabet calligraphie %
%%%%%%%%%%%%%%%%%%%%%%%%%
\def\cA{{\mathcal A}}
\def\cB{{\mathcal B}}
\def\cC{{\mathcal C}}
\def\cD{{\mathcal D}}
\def\cE{{\mathcal E}}
\def\cF{{\mathcal F}}
\def\cG{{\mathcal G}}
\def\cH{{\mathcal H}}
\def\cI{{\mathcal I}}
\def\cJ{{\mathcal J}}
\def\cK{{\mathcal K}}
\def\cL{{\mathcal L}}
\def\cM{{\mathcal M}}
\def\cN{{\mathcal N}}
\def\cO{{\mathcal O}}
\def\cP{{\mathcal P}}
\def\cQ{{\mathcal Q}}
\def\cR{{\mathcal R}}
\def\cS{{\mathcal S}}
\def\cT{{\mathcal T}}
\def\cU{{\mathcal U}}
\def\cV{{\mathcal V}}
\def\cW{{\mathcal W}}
\def\cX{{\mathcal X}}
\def\cY{{\mathcal Y}}
\def\cZ{{\mathcal Z}}
\newcommand{\rmod}[1]{\: \mbox{mod} \: #1}

\def\vr{\mathbf r}

\def\e{{\mathbf{\,e}}}
\def\ep{{\mathbf{\,e}}_p}
\def\em{{\mathbf{\,e}}_m}

\def\Tr{{\mathrm{Tr}}}
\def\Nm{{\mathrm{Nm}}}

 \def\SS{{\mathbf{S}}}

\def\lcm{{\mathrm{lcm}}}

\def\({\left(}
\def\){\right)}
\def\fl#1{\left\lfloor#1\right\rfloor}
\def\rf#1{\left\lceil#1\right\rceil}

\def\mand{\qquad \mbox{and} \qquad}

         \newcommand{\comm}[1]{\marginpar{\vskip-\baselineskip
%raise themarginpar a bit
         \raggedright\footnotesize
\itshape\hrule\smallskip#1\par\smallskip\hrule}}

%%%%%%%%%%%%%%%%%%%%%%%%%%%%%%%%%%%%%%%%%%%%%%%%%%%%%%%%
%%%%%%%%%%%%%%%%%%%%%%%%%%%%%%%%%%%%%%%%%%%%%%%%%%%%%%%%
%%%%%%%%%%%%%%%%%%%%%%%%%%%%%%%%%%%%%%%%%%%%%%%%%%%%%%%%
%%%%%%%%%%%%%%%%%%%%%%%%%%%%%%%%%%%%%%%%%%%%%%%%%%%%%%%%

%%%%%%%  END OF STANDARD STUFF %%%%%%%%%

%%%%%%%%%%%%%%%%%%%%%%%%%%%%%%%%%%%%%%%%%%%%%%%%%%%%%%%%
%%%%%%%%%%%%%%%%%%%%%%%%%%%%%%%%%%%%%%%%%%%%%%%%%%%%%%%%
%%%%%%%%%%%%%%%%%%%%%%%%%%%%%%%%%%%%%%%%%%%%%%%%%%%%%%%%
%%%%%%%%%%%%%%%%%%%%%%%%%%%%%%%%%%%%%%%%%%%%%%%%%%%%%%%
%%%%%%%%%%%
%%% Spell

\hyphenation{re-pub-lished}

\parskip 4pt plus 2pt minus 2pt

\mathsurround=1pt

\def\bfdefault{b}
\overfullrule=5pt

\def \F{{\mathbb F}}
\def \K{{\mathbb K}}
\def \Z{{\mathbb Z}}
\def \Q{{\mathbb Q}}
\def \R{{\mathbb R}}
\def \C{{\\mathbb C}}
\def\Fp{\F_p}
\def \fp{\Fp^*}

\title[Linear Equations with Rational Fractions]
{Linear Equations with Rational Fractions of Bounded Height and Stochastic 
Matrices}

 \author[I. E. Shparlinski] {Igor E. Shparlinski}

\address{Department of Pure Mathematics, University of New South Wales,
Sydney, NSW 2052, Australia}
\email{igor.shparlinski@unsw.edu.au}

%% \date{ }

\begin{abstract}  
We obtain a tight, up to a logarithmic factor, upper bound
on the number of solutions to the equation 
$$
\sum_{j=1}^n a_j \frac{s_j}{r_j} =a_0, \qquad  
$$
with variables $r_1,\ldots,r_n$ in an arbitrary 
box at the origin and variables $s_1,\ldots, s_n$
in an essentially arbitrary translation of this box.
We apply this result to get an upper bound on the number 
of stochastic matrices with rational entries of bounded height. 
%For  a positive integer   $H$, we consider the set 
%$$
%\cF(H)  = \{  s/r~:~r,s \in \Z,\ \gcd(r,s) = 1,
% \ H\ge r > s \ge 1\}
%$$ 
%of Farey fractions of order $H$ and estimate the number of solutions 
%to linear equations and congruences with variables from $\cF(H)$.
%In particular, we obtain an almost tight upper bound on the number of stochastic 
%matrices. 
\end{abstract}


\maketitle

 \section{Introduction}

 \subsection{Motivation} 
 
 We recall that a {\it stochastic 
matrix\/} is a square $n \times n$ matrix $A=(\alpha_{i,j})_{i,j=1}^n$
with nonnegative entries and such that 
\begin{equation}
\label{eq:Stoch Mat}
\sum_{j=1}^n \alpha_{i,j} =1, \qquad i = 1, \ldots, n.
\end{equation}

Furthermore, as usual, for a rational number $\alpha$ 
we define its height $\hh(\alpha)$  as 
$\max\{|s|, r\}$ where the integers $r,s \in \Z$
are uniquely defined by the conditions
$\alpha = s/r$,  $\gcd(r,s) = 1$,  $r \ge 1$. 

We now define $S_n(H)$ as the number of stochastic 
$n \times n$ matrices with rational entries of height at most 
$H$, that is, with entries from the set 
$$
\cF(H)  = \{\alpha\in \cQ  : \alpha \ge 0, \ \hh(\alpha) \le H\}.
$$ 
In particular, $\cF(H) \cap [0,1]$ is the classical 
set of {\it Farey fractions\/}  of order $H$.

The question of estimating of $S_n(H)$ seems to be quite 
natural, however it has never  been addressed in the literature. 
Since the conditions~\eqref{eq:Stoch Mat} are independent, we
clearly have 
\begin{equation}
\label{eq:S and L}
S_n(H) = L_n(H)^n, 
\end{equation}
where $L_n(H)$ is the number of solutions to the
linear equation 
$$
\sum_{j=1}^n \alpha_{j} =1, \qquad \alpha_j \in \cF(H), \ j =1, \ldots, n,
$$
which we also write as 
\begin{equation}
\label{eq:Lin Eq 1}
\sum_{j=1}^n  \frac{s_j}{r_j} =1 \qquad   H \ge r_j > s_j \ge 0,
\ \gcd(r_j,s_j) = 1, \ j =1, \ldots, n.  
\end{equation}

We remark that each solution to~\eqref{eq:Lin Eq 1} leads to an 
integer  solution 
to 
\begin{equation}
\label{eq:Lin Eq Hom}
\sum_{j=1}^n  \frac{m_j}{k_j} =0, \qquad   |k_j|, |m_j| \le K,
\end{equation}
with $K = 2H$. Indeed, it is enough to set $(k_j,m_j)=(r_j,s_j)$, $j=1, \ldots, n-1$, 
and $(k_n, m_n) = (r_n, s_n-r_n)$. Furthermore, distinct solutions to~\eqref{eq:Lin Eq 1} 
yields distinct solutions to~\eqref{eq:Lin Eq Hom}.  

We now recall that for $n=3$, the result of Blomer, Br{\" u}dern 
and Salberger~\cite{BBS}
gives an asymptotic formula  $K^3Q(K) + O(K^{3-\delta})$
for the number of solutions to~\eqref{eq:Lin Eq Hom}, where  
 $Q \in \Q[X]$ is a polynomial of degree $4$ and $\delta > 0$ 
is some absolute constant. 
This immediately implies the bound
$$
L_3(H) = O(H^3 (\log H)^4). 
$$
It is also noted in~\cite{BBS} that the same 
method is likely to work  for any $n$, it may also work for the 
equation~\eqref{eq:Lin Eq 1} 
directly and give an asymptotic formula for $L_n(H)$.
The more elementary approach of  Blomer and Br{\" u}dern~\cite{BlBr}
can probably be used to get an upper bound on $L_n(H)$ of the 
right order of magnitude. 
However working out the  above approaches 
from~\cite{BlBr,BBS} in full 
detail may require significant efforts.  Here we suggest an alternative way to
estimate $L_n(H)$ via modular reduction modulo an appropriate prime and bounds on some double exponential 
sums with rational fractions. Although the bound obtained via this approach does not reach the same strength 
as the hypothetical results that can be derived via the methods of 
 Blomer, Br{\" u}dern and Salberger~\cite{BBS}
or  Blomer and Br{\" u}dern~\cite{BlBr},  it is weaker only by a 
 power of a logarithm. 
% ,  more precisely, it has $ (\log H)^{2^n - 1 + o(1)}$ instead of
%hypothetical $ (\log H)^{2^n - n- 1 }$.
 
 Furthermore, the suggested here method seems to be more robust and also applies to 
 more general equations, see~\eqref{eq:Lin Eq a} below. 
 Besides it also works for variables in distinct intervals and 
 not necessary at the origin.  However it is unable to produce asymptotic 
 formulas.  

Clearly, to estimate  $S_n(H)$ 
one can 
use general bounds for the number of integral points
on hypersurfaces, see, for example,~\cite{Marm} and references therein,
however they do no reach the strength our results.

\subsection{Main results}

We 
derive the bounds $L_n(H)$ and thus on $S_n(H)$ from a bound on 
the number of solutions of 
a much more general linear equation. 

Namely, we consider two boxes  of the form
\begin{equation}
\label{eq:box B0}
 \fB_0 = [1, H_1]\times \ldots \times [1, H_n]
\end{equation}
and 
\begin{equation}
\label{eq:box B}
 \fB =[B_1+1,B_1+H_1]\times \ldots \times [B_n+1, B_n+H_n],
\end{equation}
with arbitrary integers  $A_j$, $B_j$ and 
positive integers  $H_j$, $i=1, \ldots, n$. 
We note that the boxes $\fB_0$ and $\fB$ are of the same dimensions 
but  $\fB_0$ is positioned at the origin, while $\fB$ can be at an 
arbitrary location in $\R^n$.

Namely, for a vector  a vector
 $\vec{a} = (a_0, a_1, \ldots, a_n) \in \Z^{n+1}$
we use  $N_n(\vec{a}; \fB_0,\fB)$ to denote the number of solutions to the
equation 
\begin{equation}
\label{eq:Lin Eq a}
\sum_{j=1}^n a_j \frac{s_j}{r_j} =a_0, \qquad  
\qquad 
(r_1,\ldots,r_n) \in \fB_0, \ 
(s_1,\ldots,s_n) \in \fB.
\end{equation}
We remark that we have droped the condition of co-primality of the variables,
which only increases the number of solutions, but does not affect our 
main results.

Throughout the paper, any implied constants in the symbols $O$,
$\ll$  and $\gg$  may depend on the real  parameter $\varepsilon> 0$ and
the integer parameter $n\ge 1$. We recall that the
notations $U = O(V)$, $U \ll V$ and  $V \gg U$ are all equivalent to the
statement that the inequality $|U| \le c V$ holds with some
constant $c> 0$.

\begin{theorem}
\label{thm:NnaH} 
Let  $\fB$ and $\fB_0$ be two boxes of the form~\eqref{eq:box B0}
and~\eqref{eq:box B}, respectively, with $\max_{j=1, \ldots, n} H_j = H$. 
For any vector $\vec{a} = (a_0, a_1, \ldots, a_n) \in \Z^{n+1}$
with $1 \le |a_i| \le \exp\(H^{O(1)}\)$, $i=1, \ldots,n$, we
have 
$$
N_n(\vec{a}; \fB_0,\fB)\le H_1\ldots H_n (\log H)^{2^n - 1 + o(1)}.
$$
\end{theorem}

We remark that in Theorem~\ref{thm:NnaH}  
there is no restriction on the size of non-vanishing 
of $a_0$ on the right hand size of~\eqref{eq:Lin Eq a}. 

\begin{cor}
\label{cor:SnH} 
We have 
$$
 H^{n^2} \ll S_n(H) \le H^{n^2} (\log H)^{n(2^n - 1) + o(1)}. 
$$
\end{cor}

We derive these results via modular reduction of the 
equation~\eqref{eq:Lin Eq a}
modulo an appropriately chosen prime $p$. In turn, to estimate the number of solutions
of the corresponding congruence, we use a new  bound of double exponential sums, 
which slightly improves~\cite[Lemma~3]{Shp}. We present this result in 
a larger generality than is need for proving Theorem~\ref{thm:NnaH} as
we believe it may be of independent interest. 

%%%%%%%%%
Furthermore, since there is a gap between upper and lower bounds of 
Corollary~\ref{cor:SnH} it is natural to do some numerical experiments 
and try to understand the asymptotic behaviour of  $S_n(H)$ as $H \to \infty$.
Thus, it is interesting to design an efficient algorithm to compute $S_n(H)$. 
Since by~\eqref{eq:S and L},  it is enough to compute $L_n(H)$
we see that one can do this via the following naive algorithm:
for each choice of $n-1$ positive integers $r_i \le H$, $i = 1, \ldots, n-1$,  then  each choice of 
$n-1$ nonnegative integers $s_i < r_i$, with $\gcd(s_i,r_i)=1$,  
$i = 1, \ldots, n-1$, check whether
$$
1- s_1/r_1 - \ldots - s_{n-1}/r_{n-1} \in \cF(H).
$$ 
This obviously gives an algorithm of complexity $H^{2n-2 + o(1)}$. 
We now show that one can compute $S_n(H)$ faster.

\begin{theorem}
\label{thm:SnH Alg} 
There is a deterministic algorithm to compute $S_n(H)$ 
 of complexity $H^{3n/2-1 + o(1)}$. 
\end{theorem}

\section{Preliminaries}

\subsection{Background on totients}

Let $\varphi(r)$ denote the Euler function.

We  need the following well-known consequence of the sieve of Eratosthenes.

\begin{lemma}
\label{lem:erat} For any integers $r \ge \ell\ge 1$,
$$
\sum_{\substack{s=1 \\ \gcd(s,r)=1}}^\ell 1 = \frac{\varphi(r)}{r}\ell
+ O(r^{o(1)}). 
$$
\end{lemma}

\begin{proof}  For an integer $d \ge 1$ we use $\mu(d)$ to denote the M\"obius function. 
We recall that $\mu(1) = 1$, $\mu(d) = 0$ if $d \ge 2$ is not square-free,  and $\mu(d) = (-1)^{\omega(d)}$  otherwise, where $\omega(d)$ is the number of prime
divisors of $d$.

Now, using the M\"obius function $\mu(d)$ over the divisors of
$r$ to detect
the co-primality condition and interchanging the order of summation, we
obtain  
$$
\sum_{\substack{s=1 \\ \gcd(s,r)=1}}^\ell 1 = \sum_{d\mid r}\mu(d)
\fl{\frac{\ell}{d}} 
 =
\ell\sum_{d\mid r}\frac{\mu(d)}{d}+O\left(\sum_{d\mid r}|\mu(d)|\right).
$$

We now  use the well-known  identity
$$
\sum_{d \mid r}\frac{\mu{(d)}}{d}=\frac{\varphi{(r)}}{r},
$$
see~\cite[Section~16.3]{HaWr} and 
also that 
and that
$$
\sum_{d \mid r}  |\mu(d)| = 2^{\omega(r)},
$$
see~\cite[Theorem~264]{HaWr}, which yield 
$$
\sum_{\substack{s=1 \\ \gcd(s,r)=1}}^\ell 1  =  \frac{\varphi(r)}{r}\ell + O\(2^{\omega(r)}\).
$$
Since obviously $\omega(r)! \le r$, the result now follows immediately.
\end{proof}

One can certainly obtain a much more precise version of the following 
statement, which we present in a rather crude form that is, however, 
sufficient for our applications. 

\begin{lemma}
\label{lem:EulerMoment} For any  $n \ge 1$ we have 
$$
\sum_{r=1}^H  \varphi(r)^{n-1} \gg H^{n}.
$$
\end{lemma}

\begin{proof}  By the   H{\"o}lder inequality, 
$$
\(\sum_{r=1}^H  \varphi(r)\)^{n-1} \le H^{n-2} \sum_{r=1}^H  \varphi(r)^{n-1}.
$$
Using the classical asymptotic formula
$$
\sum_{r=1}^H  \varphi(r) = \(\frac{3}{\pi^2}+o(1)\) H^2, 
$$
see~\cite[Theorem~330]{HaWr},
we conclude the proof. 
\end{proof}

\subsection{Background on divisors}

For an integer $m\ge 1$, we use  $\tau(m)$ to denote  the divisor 
function
$$
\tau(m) = \#\{d \in \Z~:~d\ge 1, \ d\mid m\}
$$
 and $\Delta(m)$ to denote  the Hooley 
function 
$$
\Delta(m) = \max_{u \ge 0} \#\{d \in \Z~:~u < d \le eu, \ d\mid m \}.
$$

We need upper bounds on the average values of $\tau(m)$ and $\Delta(m)$, 
where $m$ runs through terms of arithmetic progressions indexed by prime numbers. 
We derive these bounds from very general results of Nair and Tenenbaum~\cite{NaTe}.

We start with the $\tau(m)$.

\begin{lemma}
\label{lem:Aver tau} For any fixed real  $\varepsilon > 0$ and integer $n \ge 1$, 
for positive integers $a$, $m$ and $M \ge \max\{a,m^{\varepsilon}\}$, 
we have 
$$
\sum_{\substack{M \le p \le 2M \\p~\mathrm{prime}}} \tau(a+pm)^n 
\ll M (\log M)^{2^n - 2}.
$$
\end{lemma}

\begin{proof}  We apply~\cite[Theorem~3]{NaTe} (taken with the polynomials
$Q(X) = mX +a $ and $x=y=M$), and derive
$$
\sum_{\substack{M \le p \le 2M \\p~\mathrm{prime}}} \tau(a+pm)^n 
\ll \frac{M}{\log M} \prod_{\substack{p \le M \\p~\mathrm{prime}}}\(1-\frac{1}{p}\)
\sum_{m \le M} \frac{\tau(m)^n}{m}, 
$$
where the implied constant depends only on  $\varepsilon$ and $n$.
Using the Mertens formula for the product over primes, and 
also the classical bound of Mardjanichvili~\cite{Mard}
\begin{equation}
\label{eq:Mardj}
\sum_{m \le M} \tau(m)^n \ll M (\log M)^{2^n - 1}
\end{equation}
(combined with partial summation), we easily derive the result.
\end{proof}

For $\Delta(m)$ we obtain a slightly stronger bound.

\begin{lemma}
\label{lem:Aver Delta} For any fixed real  $\varepsilon > 0$ and integer $n \ge 1$, 
for positive integers $a$, $m$ and $M \ge \max\{a,m^{\varepsilon}\}$, 
we have 
$$
\sum_{\substack{M \le p \le 2M \\p~\mathrm{prime}}} \Delta(a+pm)^n 
\ll  M (\log M)^{2^n - n-2 + o(1)}.
$$
\end{lemma}

\begin{proof}  As in the proof of Lemma~\ref{lem:Aver tau}, 
by~\cite[Theorem~3]{NaTe} we have 
$$
\sum_{\substack{M \le p \le 2M \\p~\mathrm{prime}}} \Delta(a+pm)^n 
\ll \frac{M}{\log M} \prod_{\substack{p \le M \\p~\mathrm{prime}}}\(1-\frac{1}{p}\)
\sum_{m \le M} \frac{\Delta(m)^n}{m}, 
$$
where the implied constant depends only on  $\varepsilon$ and $n$.
Now,  using~\cite[Lemma~2.2]{Ten} instead of~\eqref{eq:Mardj} we conclude the proof.  
\end{proof}

We remark that using the full power of~\cite[Theorem~3]{NaTe} 
and~\cite[Lemma~2.2]{Ten}, one can replace $o(1)$ in the power of $\log M$
in Lemma~\ref{lem:Aver Delta} by a more precise and explicit 
function of $M$.

 \subsection{Product and least common multiples of several integers}
 
We need the following result of Karatsuba~\cite{Kar}

For an integer $n \ge 1$ and real $R_1, \ldots, R_n$, let
 $J_n(R_1, \ldots, R_n)$ denote the number of solutions to congruence 
\begin{equation}
\label{eq:lcm}
r_1\ldots r_n \equiv 0 \pmod {\lcm[r_1^2\ldots r_n^2]}
\end{equation}
in positive integer $r_j \le R_j$, $j=1, \ldots, n$.

\begin{lemma}
\label{lem:lcm div} 
For any real $R_1, \ldots, R_n \ge 2$, we have 
$$
J_n(R_1, \ldots, R_n) \ll R^{1/2} (\log R)^{n^2},
$$
where $R = R_1 \ldots R_n$. 
\end{lemma}

 \subsection{Exponential sums with ratios} 
For a prime $p$, we denote $\ep(z) = \exp(2 \pi i z/p)$.
Clearly for any $p > u,v\ge 1$ the expression $\ep(av/u)$
is correctly  defined (as $\ep(aw)$
for $w \equiv  v/u \pmod p$). 

The following result is a variation of~\cite[Lemma~3]{Shp}, 
where the additional averaging over primes allows us to replace 
$Q^{o(1)}$ with a power of $\log Q$.

\begin{lemma}
\label{lem:Sum Triple} Let $Q> U, V\ge 1$ be arbitrary integers and
let $\cC\subseteq [0,U]\times [0,V]$
be an arbitrary convex domain.   Then, 
uniformly over the  integers $a\in [1, 2Q]$, 
we have 
$$
\sum_{\substack{Q \le p \le 2Q\\p~\mathrm{prime}\\\gcd(a,p)=1}} \left|
 \sum_{(u,v) \in \cC} \ep(av/u)\right|^n \ll  (U+V)^n Q (\log Q)^{2^n- 2 + o(1) }. 
$$
\end{lemma}

\begin{proof}
Since $\cC$ is convex,  for each $v$ we there  are
 integers $V \ge W_u > V_u \ge 0$ such that 
$$
\sum_{(u,v) \in \cC} \ep(av/u) =
\sum_{u=1}^U \sum_{v=V_u+1}^{W_u} \ep(av/u).
$$
Following the proof of~\cite[Lemma~3]{Shp}, we
define 
$$I = \fl{\log(2Q/V)} \mand J = \fl{\log (2Q)}. 
$$

Furthermore, for a rational number $\alpha = v/u$ with $\gcd(u,p)=1$,
we denote by $\rho_p(\alpha)$ the unique integer $w$ with $w \equiv v/u \pmod p$
and $|w| < p/2$. 
Then~\cite[Equation~(1)]{Shp} implies 
$$
\sum_{(u,v) \in \cC} \ep(av/u)\ll V R_{p}   
+ Q \sum_{j = I+1}^J T_{j,p} e^{-j},
$$
where
\begin{equation*}
\begin{split}
&R_{p} =\# \left\{u~:~1 \le u \le U,  \  
 |\rho_p(a/u)| < e^{I} \right\},\\
&T_{j,p} =\# \left\{u~:~1 \le u \le U,  \  
e^j \le |\rho_p(a/u)| < e^{j+1} \right\}.
\end{split}
\end{equation*}

Thus, using the H{\"o}lder inequality twice, we obtain:
\begin{equation*}
\begin{split}
\left|\sum_{(u,v) \in \cC} \ep(av/u)\right|^n& \ll V^n R_{p}^n    
+ Q^n \(\sum_{j = I+1}^J T_{j,p}  e^{-j}\)^n\\
& \ll V^n R_{p}^n    
+ Q^n  (\log Q)^{n-1} \sum_{j = I+1}^J T_{j,p}^n e^{-jn}. 
\end{split}
\end{equation*}
Hence
\begin{equation}
\begin{split}
\label{eq:W prelim 1}
\sum_{\substack{Q \le p \le 2Q\\p~\mathrm{prime}\\\gcd(a,p)=1}} &\left|\sum_{(u,v) \in \cC} \ep(av/u)\right|^n\\
& \qquad  \ll V^n \fR  
+ Q^n (\log Q)^{n-1} \sum_{j = I+1}^J\fT_{j} e^{-jn}, 
\end{split}
\end{equation}
where
$$
\fR = \sum_{\substack{Q \le p \le 2Q\\p~\mathrm{prime}\\\gcd(a,p)=1}}   R_{p}^n  \mand 
\fT_j = \sum_{\substack{Q \le p \le 2Q\\p~\mathrm{prime}\\\gcd(a,p)=1}}   T_{j,p}^n  .
$$

As in~\cite{Shp},  we note that if $e^j \le |\rho_p(a/u)| < e^{j+1}$, 
then $u z \equiv a \pmod p$ for some integer $z$ with 
$e^j < |z| <   e^{j+1}$. Thus $uz = a + pk$ for some integer $k$ with 
$|k| \le  K_j$ where $K_j = \fl{e^{j+1}U/Q} +1$. 
Therefore, recalling the definitions of the $\tau(m)$ and $\Delta(m)$, 
we conclude
%\begin{equation}
%\label{eq:R T tau}
$$
R_{p} \le  \sum_{|k| \le K_{I-1}} \tau\(|a+pk|\) \mand 
T_{j,p}  \le \sum_{|k| \le K_j} \Delta\(|a+pk|\).
$$
%%\end{equation}
Now using the H{\"o}lder inequality, 
changing the order of summation and applying Lemma~\ref{lem:Aver tau}
(which applies as $Q \gg K_J \ge |k|$) we derive
\begin{equation*} 
\begin{split}
\fR & \le  \sum_{\substack{Q \le p \le 2Q\\p~\mathrm{prime}}}  
 \(\sum_{|k| \le K_I}  \tau\(|a+pk|\)\)^n \\
 & \le K_I^{n-1}  \sum_{|k| \le K_I} \sum_{\substack{Q \le p \le 2Q\\p~\mathrm{prime}}}  
 \tau\(|a+pk|\)^n \ll  K_I^{n}  Q (\log Q)^{2^n-2}.
\end{split}
\end{equation*}
Similarly,  applying Lemma~\ref{lem:Aver Delta}, we see that
\begin{equation*} 
\begin{split}
\fT & \le  \sum_{\substack{Q \le p \le 2Q\\p~\mathrm{prime}}}  
 \(\sum_{|k| \le K_j}  \Delta\(|a+pk|\)\)^n \\
 & \le K_j^{n-1}  \sum_{|k| \le K_j} \sum_{\substack{Q \le p \le 2Q\\p~\mathrm{prime}}}  
\Delta\(|a+pk|\)^n \ll  K_j^{n}  Q (\log Q)^{2^n -n -2+o(1)}, 
\end{split}
\end{equation*}
for $j = I+1, \ldots, J$.
Substituting this bound in~\eqref{eq:W prelim 1}, yields
\begin{equation} 
\label{eq:W prelim 2}
\begin{split}
& \sum_{\substack{Q \le p \le 2Q\\p~\mathrm{prime}\\\gcd(a,p)=1}}  \left|
 \sum_{(u,v) \in \cC} \ep(av/u)\right|^n \\
 & \quad \ 
 \ll V^n   K_I^{n}  Q (\log Q)^{2^n-2}
 + Q^{n+1} (\log Q)^{2^n  -3 + o(1)} \sum_{j = I+1}^J
e^{-jn}   K_j^{n}   .
\end{split}
\end{equation}

We now have
$$
 K_I^{n} \ll  e^{In} U^nQ^{-n} + 1 \ll U^nV^{-n} +  1
$$
and also 
\begin{equation*} 
\begin{split}
\sum_{j = I+1}^J
e^{-jn} K_j^{n} & \ll 
\sum_{j = I+1}^J\(e^{jn}U^nQ^{-n} + 1\) e^{-jn}\\
& \ll JU^n/Q^n + e^{-In} 
\ll U^nQ^{-n} \log Q + V^nQ^{-n}.
\end{split}
\end{equation*}
Combining the above bounds with~\eqref{eq:W prelim 2},
after simple calculations, 
we obtain the desired result.
\end{proof}

\section{Proofs of Main Results}

\subsection{Proof of Theorem~\ref{thm:NnaH}}

We note that since 
$|a_1\ldots a_n| = \exp\(H^{O(1)}\)$, 
this product has at most $H^{O(1)}$ prime divisors. 
Hence, there is a constant $C>0$ such that for $Q = \rf{H^C}$
there is a set $\cP$ of
at least 
\begin{equation} 
\label{eq:P large}
\# \cP \ge 0.5Q/\log Q
\end{equation}
primes $p \in [Q,2Q]$ that are relatively 
prime with $a_1\ldots a_n$. We also assume that $C> n$.
In particular, 
\begin{equation} 
\label{eq:Q large}
H^C \gg Q  > H^n.
\end{equation} 

Let $M_n(\vec{a};p, \fB_0,\fB)$ be the number of solutions to the
congruence 
$$
\sum_{j=1}^n a_j \frac{s_j}{r_j} \equiv a_0 \pmod p,  \qquad  (r_1,\ldots,r_n) \in \fB_0,\ (s_1,\ldots,s_n)  \in \fB. 
$$

Using the orthogonality of exponential functions, we write 
$$
M_n(\vec{a};p, \fB_0,\fB)
  = \ssum_{\substack{(r_1,\ldots,r_n) \in \fB_0\\
(s_1,\ldots,s_n)  \in \fB}} \, \frac{1}{p}
\sum_{\lambda=0}^{p-1} \ep\(\lambda\(\sum_{j=1}^n a_j \frac{s_j}{r_j}- a_0 \)\).
$$
Changing the order of summation, gives the identity
$$
M_n(\vec{a};p, \fB_0,\fB) =  \frac{1}{p} \sum_{\lambda=0}^{p-1} 
 \ep\(-\lambda a_0\) \prod_{j=1}^n \sum_{r_j=1}^{H_j} 
\sum_{s_j = B_j+1}^{B_j+H_j} 
\ep\( \lambda a_j  s_j/r_j\).
$$
Now, the contribution from $\lambda=0$ gives the main term $(H_1\ldots H_n)^{2}/p$.
Extending the summation over $\lambda$ to all positive integers  $\lambda \le 2Q$ 
with $\gcd(\lambda,p)=1$, for every $p \in \cP$ we obtain 
\begin{equation*}
\begin{split}
 M_n(\vec{a};p&, \fB_0,\fB)\\ 
 & \le \frac{(H_1\ldots H_n)^{2}}{Q} +
 \frac{1}{Q} \sum_{\substack{\lambda=1\\ \gcd(\lambda, p)=1}}^{2Q} 
\prod_{j=1}^n
\left|\sum_{r_j=1}^{H_j} \sum_{s_j = B_j+1}^{B_j+H_j} \ep\(\lambda a_j  s_j/r_j\)\right|.
\end{split}
\end{equation*}
Hence, summing over all $p \in \cP$ 
and denoting
$$
W(\lambda)
=    \sum_{\substack{p\in \cP\\ \gcd(\lambda, p)=1}}^{2Q} 
\prod_{j=1}^n
\left|\sum_{r_j=1}^{H_j} \sum_{s_j = B_j+1}^{B_j+H_j} \ep\(\lambda a_j  s_j/r_j\)\right|.
$$
we see that
\begin{equation} 
\label{eq:M and W}
\sum_{p\in \cP}  M_n(\vec{a};p, \fB_0,\fB)\le \#\cP  \frac{(H_1\ldots H_n)^{2}}{Q}  + 
 \frac{1}{Q} \sum_{\substack{\lambda=1\\ \gcd(\lambda, p)=1}}^{2Q} W(\lambda).
\end{equation}
Using the H{\"o}lder inequality, we obtain
$$
W(\lambda)^n 
\le   \prod_{j=1}^n  \sum_{\substack{p\in \cP\\ \gcd(\lambda, p)=1}}^{2Q} 
\left|\sum_{r_j=1}^{H_j} \sum_{s_j = B_j+1}^{B_j+H_j}  \ep\(\lambda a_j  s_j/r_j\)\right|^n.
$$
We now invoke Lemma~\ref{lem:Sum Triple} and see that 
$$
W(\lambda) 
\ll H_1\ldots H_n Q (\log Q)^{2^n - 2 + o(1)},
$$
which together with~\eqref{eq:M and W} implies 
\begin{equation*}
\begin{split}
 \sum_{p\in \cP}  M_n(\vec{a}&;p, \fB_0,\fB) \ll  \#\cP  \frac{(H_1\ldots H_n)^{2}}{Q} + H_1\ldots H_n Q (\log Q)^{2^n - 2 + o(1)}.
\end{split}
\end{equation*} 
Hence, by~\eqref{eq:P large}, there is a prime $p\in \cP$
with 
$$
 M_n(\vec{a};p, \fB_0,\fB) 
\ll   \frac{(H_1\ldots H_n)^{2}}{Q} + H_1\ldots H_n  (\log Q)^{2^n - 1 + o(1)}.
$$
Recalling~\eqref{eq:Q large}, we see that the second term dominates and 
$\log Q$ can be replaced with $\log H$. Therefore,
%\begin{equation} 
%\label{eq:Av MapH}
$$
 M_n(\vec{a};p, \fB_0,\fB) 
\ll   H_1\ldots H_n  (\log H)^{2^n - 1 + o(1)}.
$$
%\end{equation}
Using the trivial bound $N_n(\vec{a}; \fB_0,\fB)  \le 
M_n(\vec{a};p, \fB_0,\fB)$,  we 
conclude the proof.

\subsection{Proof of Corollary~\ref{cor:SnH}}

The upper bound is immediate from Theorem~\ref{thm:NnaH}
and the equation~\eqref{eq:S and L} (we note that
for the typographical simplicity 
the values $s_j=0$ are excluded in Theorem~\ref{thm:NnaH}, 
but  a simple inductive argument allows us 
to include them).

To see the lower bound we note that by Lemma~\ref{lem:erat}
for any positive integer $r \le H$ and  $\ell = \fl{r/(n-1)}$ and
we can  choose positive integers $s_1, \ldots, s_{n-1} \le \ell$
with 
\begin{equation} 
\label{eq:Coprime}
\gcd(s_1  \ldots s_{n-1}, r)=1 
\end{equation} 
in
$$
\((1+o(1))\frac{\varphi(r)}{r} \ell\)^{n-1} 
\gg  \varphi(r)^{n-1}
$$
possible ways. After this we set $s_n = r - s_1 - \ldots - s_{n-1}$.
Clearly, by the co-primality condition~\eqref{eq:Coprime}, 
the vectors of rational numbers $(s_1/r, \ldots, s_n/r)$ 
obtained via the above construction 
are pairwise distinct. Hence 
$$
L_n(H) \gg \sum_{r=1}^H \varphi(r)^{n-1}
$$
Applying Lemma~\ref{lem:EulerMoment} we obtain $L_n(H) \gg H^n$ 
and the result follows from~\eqref{eq:S and L}.  

\subsection{Proof of Theorem~\ref{thm:SnH Alg}}
Clearing the denominators, we transform~\eqref{eq:Lin Eq 1}, 
into the following equation
$$
\sum_{j=1}^n s_j \frac{r_1\ldots r_n}{r_j}
 =r_1\ldots r_n . 
$$
Since for every $j =1, \ldots, n$ we have  $\gcd(r_j,s_j) = 1$, 
we see the divisibility 
$$
r_j \mid\frac{r_1\ldots r_n}{r_j}
$$
or $r_j^2 \mid  r_1\ldots r_n$,
which implies the congruence~\eqref{eq:lcm} for every 
solution to~\eqref{eq:Lin Eq 1}. 

Using standard arithmetic algorithms for computing the 
greatest common divisor and this the least common multiple, we see that we
can enumerate all vectors of positive integers
$(r_1, \ldots, r_n)$ for which~\eqref{eq:Lin Eq 1}
has a solution for some $s_1, \ldots, s_n$ in time $O\(H^{n+o(1)}\)$.
By Lemma~\ref{lem:lcm div}, the resulting list contains 
$O(H^{n/2+o(1)})$ vectors. 

Now for each $(r_1, \ldots, r_n)$
we choose $0 \le s_j< r_j$, $j=1, \ldots, n-1$ in $O\(H^{n-1}\)$ ways, 
define $s_n$ by the equation~\eqref{eq:Lin Eq 1} and check whether 
other conditions in~\eqref{eq:Lin Eq 1}. This leads to the desired 
algorithm. 

\section{Comments}

We remark that estimating $L_n(H)$ by the number of solutions to 
an equations of the type~\eqref{eq:Lin Eq a} (that is, without the 
co-primality condition) can lead to additional logarithmic losses.
This effect  has been mentioned in~\cite{BBS} and can also be easily 
seen for $n=2$. 

We recall that   a square $n \times n$ matrix $A$ is
called {\it doubly stochastic\/} if both $A$ and the transposed 
matrix $A^T$ are stochastic. 

We now define $\SS_n(H)$ as the number of doubly stochastic 
$n \times n$ matrices with rational entries from $\cF(H)$.
We have the following trivial bounds
\begin{equation}
\label{eq:SS}
H^{n^2 -2n + 2} \ll \SS_n(H) \le  H^{n^2-n} (\log H)^{(n-1)(2^n - 1) + o(1)}.
\end{equation}
Indeed the upper bound in~\eqref{eq:SS} follows immediately 
from Theorem~\ref{thm:NnaH} and the observation that if the 
top $n-1$ rows of a doubly stochastic matrix are fixed
then the last row is uniquely defined. To get a lower bound on $\SS_n(H)$,
we fix a positive integer $r$ and  for each 
$i=1, \ldots, n-1$ we choose the elements $\alpha_{i,j}= s_{i,j}/r$, 
$i=1, \ldots, n-1$, $j=1, \ldots, n$, of the first $n-1$ rows
as in the proof of Corollary~\ref{cor:SnH} (with respect to the same $r$). 
After this we also define
$$
\alpha_{n,j} = 1 - \sum_{i=1}^{n-1} \alpha_{i,j}, \qquad 
j=1, \ldots, n.
$$
It only remains to note that 
$$
\sum_{j=1}^{n} \alpha_{n,j} =
n - \sum_{j=1}^{n} \sum_{i=1}^{n-1} \alpha_{i,j}
= n -\sum_{i=1}^{n-1}  \sum_{j=1}^{n}  \alpha_{i,j} = 1.
$$
Now, simple counting yields the lower bound in~\eqref{eq:SS}. 

%We remark that one can probably improve lower bounds
%in Corollary~\ref{cor:SnH} and in~\eqref{eq:SS}
%via considering rows of the form
%$$
%\(s_1/r_1, \ldots, s_{n}/r_n\)
%$$
%where $r_1, \ldots, r_{n-1}$ are divisors of some 
%positive integer $r\le H$, $s_j$ are chosen with 
%$\gcd(s_j,r_j) = 1$ and $0 \le s_j < r_j/(n-1)$, 
%$j = 1, \ldots, n-1$, after which $s_n$ and $r_n$ 
%are uniquely defined. 

%
%For $n=3$, the result of Blomer, Br{\" u}dern and Salberger~\cite{BBS}
%gives an asymptotic formula for the number of solutions to~\eqref{eq:Lin Eq a}
%with $a_1=a_2 =a_3=1$ and $a_0=0$. It is also noted in~\cite{BBS} that the same 
%method works for any $n$ and also for the relatively prime variables
%$\gcd(r_j,s_j)=1$. However it is not clear whether this method works 
%for   equations with $a_0 \ne 0$. 

%\hrule 
%
%Tim's bounds are better for $n \le 4$?
%
%Any nontrivial lower bounds?

%
%Instead of the $\tau$, one can work with Hooley's $\Delta$-function 
%as we look for divisors $z \in [e^j, e^{j+1}$.
%But the we need a bound on 
%$$
%\sum_{m \le Q} \Delta(m)^n,
%$$
%which is better than that for $\tau$. 
% 
\section*{Acknowledgement}

The author is grateful to Valentin Blomer and Tim Browning for very 
useful discussions. 
 This work was  supported in part by ARC Grant~DP140100118.

\end{document}